\documentclass[review,onefignum,onetabnum]{siamart171218}

\usepackage{lipsum}
\usepackage{amsfonts}
\usepackage{graphicx}
\usepackage{epstopdf}
\usepackage{algorithmic}
\usepackage{multirow}
\usepackage{amsmath}
\nolinenumbers

\ifpdf
\DeclareGraphicsExtensions{.eps,.pdf,.png,.jpg}
\else
\DeclareGraphicsExtensions{.eps}
\fi

\newcommand{\vep}{\varepsilon}
\newcommand{\Scal}{\mathcal{S}}
\newcommand{\Pcal}{\mathcal{P}}
\newcommand{\Acal}{\mathcal{A}}
\newcommand{\Ical}{\mathcal{I}}
\newcommand{\Usf}{\mathsf{U}}
\newcommand{\Vsf}{\mathsf{V}}

\newcommand{\Asf}{\mathsf{A}}
\newcommand{\Ysf}{\mathsf{Y}}
\newcommand{\Qsf}{\mathsf{Q}}
\newcommand{\Bsf}{\mathsf{B}}

\newcommand{\wt}[1]{\widetilde{#1}}

\newsiamremark{remark}{Remark}
\newsiamremark{hypothesis}{Hypothesis}
\crefname{hypothesis}{Hypothesis}{Hypotheses}
\newsiamthm{claim}{Claim}

\headers{Random sampling in Schwarz method for elliptic equation}{K. Chen, Q. Li, and S. J. Wright}

\title{Schwarz iteration method for elliptic equation with rough media based on random sampling
	\thanks{Submission DATE. \funding{The work is supported in part by the National Science Foundation via grant 1740707. The work of SW is further supported in part by National Science Foundation grants 1447449, 1628384, and 1634597; Subcontract 8F-30039 from Argonne National Laboratory; and Award N660011824020 from the DARPA Lagrange Program. The work of KC and QL is further supported in part by Wisconsin Data Science Initiative and National Science Foundation via grant DMS-1750488, and DMS-1107291: RNMS KI-Net.}}}

\author{Ke Chen\thanks{Department of Mathematics, University of Texas at Austin, Austin, TX, 78731
		(\email{kechen@math.utexas.edu}, \url{https://web.ma.utexas.edu/users/kechen/}).}
	\and Qin Li\thanks{Department of Mathematics, University of Wisconsin-Madison, Madison, WI, 53706
		(\email{qinli@math.wisc.edu}, \url{http://www.math.wisc.edu/\textasciitilde qinli/}).}
	\and Stephen J. Wright\thanks{Department of Computer Science, University of Wisconsin-Madison, Madison, WI, 53706
		(\email{swright@cs.wisc.edu}, \url{http://pages.cs.wisc.edu/\textasciitilde swright/}).}
}

 \graphicspath{{../figure/}}

\begin{document}

\maketitle

\begin{abstract}
We propose a computationally efficient Schwarz method for elliptic
equations with rough media. A random sampling strategy is used to find
low-rank approximations of all local solution maps in the offline
stage; these maps are used to perform fast Schwarz iterations in the
online stage. Numerical examples demonstrate accuracy and robustness
of our approach.
\end{abstract}



\section{Introduction}
%%%%%%%%%%%%%%%%%%%%%%%%%%%%%%%%%%%%%%%%%%%
%Introduction outline:
%1. Elliptic equations with rough media are hard.
%2. Different methods have been proposed: direct methods and iterative methods. 
%3. Discussion why Schwarz method is good.
%4. why Schwarz method is not good and why our work is important.
%5. outline of the paper
%%%%%%%%%%%%%%%%%%%%%%%%%%%%%%%%%%%%%%%%%%%

Many problems of fundamental importance in science and engineering
have structures that span several spatial scales. To give two
examples, airplane wings constructed from fiber-reinforced composite
materials and the permeability of groundwater flows modeled by porous
media can both be described by partial differential equations (PDEs)
whose coefficients have multiscale structure. Direct numerical
simulation for these problems is difficult because the discretized
system has very many degrees of freedom. Domain decomposition and
parallel computing are usually needed to solve the discretized system.

Different types of PDEs typically require different strategies to
overcome the computational difficulties arising from multiple
scales. Computations with elliptic PDEs on rough media are termed
``numerical homogenization,'' for which several approaches have been
proposed. The most well-known algorithms include the Generalized
Finite Element Method (GFEM)\cite{Osborn83,BL11}, Heterogeneous
Multiscale Methods (HMM)
\cite{abdulle_weinan_engquist_vanden-eijnden_2012,e2003}, Multiscale
Finite Element Methods (MsFEM)
\cite{efendiev2009multiscale,HOU1997169}, local orthogonal
decomposition \cite{maalqvist2014localization}, and local basis
construction \cite{Owhadi14,Owhadi15}. Most of these methods divide
the computation into offline and online stages. The offline step finds
local bases that are adaptive to local properties and capture
small-scale effects on the macroscopic solutions. In the online step,
a global stiffness matrix is assembled in such a way that this
small-scale information is implanted and preserved. The online
computation is performed on a coarse grid, thus reducing computational
costs.  Alternative approaches use a domain decomposition framdwork to
target directly the problem on the fine mesh (see, for example,
\cite{graham2007domain,mathew2008domain,toselli2006domain,dolean2015introduction,smith2004domain}
and references therein). These methods are typically iterative,
dividing the domain into patches to allow for parallel
computation. The most important issues for these approaches are (1)
the local solvers need to resolve the fine grids, which drives up
computational costs; and (2) the convergence rate depends on the
conditioning, so many iterations are required for an ill-conditioned
problem. Several strategies have been proposed to overcome these
issues. One such strategy is to use MsFEM as a preconditioner for the
Schwarz method \cite{Aarnes2002MultiscaleDD} or to construct coarse
spaces via solution of local eigenvalue problems
\cite{Efendiev2010,EFENDIEV2011}. This preconditioner differs from the
traditional domain decomposition preconditioner in that the coarse
solver is adaptive to the small scale features. In contrast to the
deterioration of traditional preconditioner when multiscale structure
is present, its adaptive counterpart is nearly independent of high
contrast and of small scales within the media
\cite{Aarnes2002MultiscaleDD,GALVIS2014456}.

We propose a rather different perspective for numerical
homogenization, under the framework of domain decomposition with
Schwarz iteration. The homogenization phenomenon for elliptic
equations with highly oscillatory media refers to the fact that its
solution can be approximated by the solution to an effective equation
that has no oscillation in the media \cite{Allaire92}.  Although it is
not always easy to identify the effective equation explicitly, just
the knowledge that
% such effective equation
it exists allows us to argue that the equation can be ``compressed''
in some sense.  One still needs to understand what aspect, exactly,
can be ``compressed.'' In our previous work \cite{chen2018random}, we
demonstrated that the collection of Green's functions, when
discretized and stored in a matrix form, can be compressed
approximately into a low-dimensional space spanned by its leading
singular vectors. To obtain these representative basis functions, we
apply a random sampling technique analogous to the one used in
compressed sensing. In effect, we explored the counterpart of the
randomized SVD (RSVD) algorithm~\cite{Tropp11} in the PDE setting, and
in the framework of GFEM, we were able to capture the representative
local $a$-harmonic functions at significantly reduced computational cost.

In this article, we argue there is another quantity that can be
``compressed,'' giving another route to higher numerical
efficiency. In the Schwarz procedure, the whole domain is decomposed
into multiple subdomains with small overlaps. At each iteration, a
local solution is obtained on each patch, using boundary conditions
for that patch supplied by its neighbors. These local solutions yield
boundary conditions for neighboring patches, which are then used in
the next round of Schwarz iteration. In effect, each local solution
procedure is a boundary-to-boundary map. Each Schwarz iteration is
contractive, so the overall procedure converges.  The total cost is
determined by the cost of local solvers, the number of subdomains, and
the number of Schwarz iterations. The latter two can be balanced by
using preconditioning techniques. The cost of the first factor --- the
boundary-to-boundary map --- can be reduced by noting that this map is
compressive. Its spectrum decays exponentially, so a compact
approximation to this map can be obtained and applied rapidly, with
accuracy sufficient to allow convergence of the outer Schwarz
procedure.  Only a few samples (in the form of randomized boundary
conditions) are needed to approximate the boundary-to-boundary
maps. These are performed in the offline stage.  Our procedure can be
regarded as a counterpart of an RSVD algorithm for a PDE solution
operator.  (Our approach in \cite{chen2018random} approximates the
range of the solution space instead.)  This procedure was discussed in
\cite{Chen19c} for the case of the radiative transfer equation, where
there are two scales both needing to be resolved. In this article, we
target an elliptic homogenization problem. Our main contribution lies
in bringing randomized sampling technique and incorporating it with
multiscale domain decomposition methods. Our solver is adaptive to
small scale features, inexpensive to build offline, and has fast
online convergence.

In the remainder of the paper, we introduce fundamental concepts
in~\cref{sec:schwarz} and describe our algorithm
in~\cref{sec:randomSampling}, first reviewing RSVD and interpreting it
in our setting. Computational results are shown
in~\cref{sec:numerics}.

\section{Schwarz method for elliptic equation}\label{sec:schwarz}

%%%%%%%%%%%%%%%%%%%%%%%%%%%%%%%%%%%%%%%%%%%
% section 2 outline:
%1. Elliptic equations with rough media
%2. domain decomposition and Schwarz method
%%%%%%%%%%%%%%%%%%%%%%%%%%%%%%%%%%%%%%%%%%%

We review briefly the elliptic equation with rough media and its
homogenization limit, then present an overview of Schwarz iteration
method under the domain decomposition framework.

\subsection{Elliptic equation with rough media}\label{subsec:elliptic}
We consider the boundary value problem for a scalar second order
elliptic equation with rough media:
\begin{equation}\label{eqn:elliptic}
\begin{cases}
\nabla\cdot (a^\vep(x)\nabla u^\vep(x) )= 0\,, & \mbox{\rm in $\Omega \subset \mathbb{R}^d$}\\
u^\vep(x) = b(x) \,, & \mbox{\rm on $\partial \Omega$,}
\end{cases}
\end{equation}
where the function $a^\vep(x)\in L^\infty(\Omega)$ models the
media. This is the typical governing equation in the modeling of water
flow in porous media and heat diffusion through composite
materials. In these examples, the media tensor $a^\vep(x)$ usually
exhibits multiscale structure; the parameter $\vep$ denotes the
smallest scale that appears explicitly in the media. The media
$a^\vep(x)$ is assumed to be uniformly bounded, namely $\alpha \le
a^\vep(x) \le \beta$. The Dirichlet boundary condition $f(x)$ is a
macroscopic quantity that usually does not contain small scales, that
is, it is independent of $\vep$. For simplicity, we restrict ourselves
to the case in which $d=2$.

For equations demonstrating certain structures, the asymptotic limit
of the equation can be derived. In particular, when the media is
pseudo-periodic, this homogenization procedure is classical. Denoting
\[
a^\vep(x) = a(x,{x}/{\vep}) = a(x,y)\,,\quad\text{with}\quad y={x}/{\vep}
\]
and assuming that $a(x,y)$ is periodic in the fast variable $y$, we
have the following theorem:

\begin{theorem}(\cite[Theorem~2.1]{moskow_vogelius_1997}])
Assuming a priori that $u^\ast \in H^2(\Omega)$, then there exists
constant $C>0$ such that
\begin{equation*}
\| u^\vep - u^\ast \|_{L^2(\Omega)} \leq C\vep \| u^\ast \|_{H^2(\Omega)}
\end{equation*}
where $u^\vep$ is the solution to equation \eqref{eqn:elliptic} and
$u^\ast$ the solution to the following effective equation
\begin{equation}\label{eqn:elliptic_homo}
\begin{cases}
\nabla\cdot (a^\ast(x)\nabla u^\ast(x) )= 0\,, & \mbox{\rm in $\Omega \subset \mathbb{R}^d$}\\
u^\ast(x) = b(x) \,,  &\mbox{\rm on $\partial \Omega$.}
\end{cases}
\end{equation}
\end{theorem}
Although it is not easy to find the smooth effective media $a^\ast$
except in some very special cases, the theorem suggests that the
asymptotic limit of the equation with highly oscillatory media is one
that has smooth media. This observation, when interpreted correctly,
can lead to significant improvements in computation. A naive
discretization method applied to \eqref{eqn:elliptic} would require
discretization with $\Delta x\ll \vep$, while in the limit, assuming
$a^\ast$ is known, the problem \eqref{eqn:elliptic_homo} could be
solved with discretization $\Delta x = o(1)$, leading to a much less
expensive computation.

Our approach exploits this observation.  In~\cref{sec:randomSampling},
we demonstrate that the boundary-to-boundary map used in Schwarz
iteration is indeed compressible, and a random sampling technique can
be used to approximate this map cheaply, leading to computational
savings. Importantly, though, our approach does not require explicit
knowledge of the limiting medium $a^\ast$.

\subsection{Domain decomposition and Schwarz iteration}\label{subsec:dd}
The Schwarz iteration procedure based on domain decomposition divides
the physical domain $\Omega$ into many
overlapping subdomains, obtains local solutions on the subdomains,
exchanges information in the form of boundary conditions for the
subdomains, and repeats the whole procedure until convergence.  We
denote by $\{\Omega_i\,, i=1,\ldots,N\}$ an open cover of $\Omega$
(see~\cref{fig:dd2}), so that
\begin{equation}\label{eqn:dd}
\Omega = \cup_{i=1}^N \Omega_i\,.
\end{equation}
\begin{figure}
	\centering
	\includegraphics[width=0.8\textwidth]{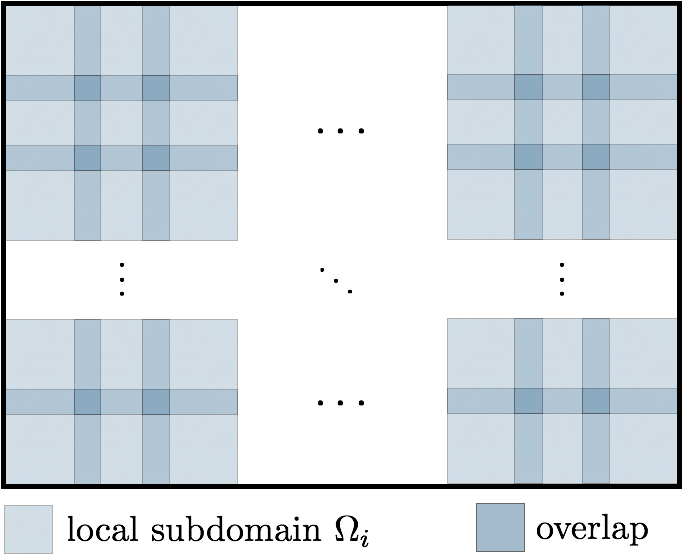}
	\caption{An overlapping domain decomposition of rectangular domain $\Omega$ in 2D}
	\label{fig:dd2}
\end{figure}
Defining the collection of indices of the subdomains that intersect
with $\Omega_i$ as follows:
\begin{equation} \label{eqn:ical}
\Ical_i = \{j\in\mathbb{Z}: \Omega_i\cap\Omega_j\neq\emptyset\}\,,
\end{equation}
the interior of each subdomain $\Omega_i$ can be written as follows:
\begin{equation}\label{eqn:interior}
	\widetilde{\Omega}_i = \Omega_i \cap \left(\cup_{j\in \Ical_i} \Omega_j^c\right)\,,
\end{equation}
where $\Omega_j^c$ denotes the complement of $\Omega_j$.  The global
solution $u^\vep$ to \eqref{eqn:elliptic} can be expressed as a
superposition of many modified local solutions:
\begin{equation}\label{eqn:global}
u^\vep(x) = \sum_{i=1}^N \eta_i(x) u^\vep_i(x)\,,
\end{equation}
where $u^\vep_i(x)$ is the local solution on $\Omega_i$, satisfying
the same elliptic equation locally over the subdomain $\Omega_i$:
\begin{equation}\label{eqn:elliptic_local}
\begin{cases}
\nabla\cdot (a^\vep(x)\nabla u_i^\vep(x) )= 0\,, & \mbox{\rm in $ \Omega_i$}\\
u_i^\vep(x) = f_i(x) \,, & \mbox{\rm on $\partial \Omega_i$,}
\end{cases}
\end{equation}
and $\eta_i(x)$, $i=1,2,\dotsc,N$ are the partition-of-unity functions
satisfying
\[
\sum_{i = 1}^N \eta_i(x)= 1\,, \quad \forall x\in\Omega\,,\quad\text{with}\quad\begin{cases}
0 \leq \eta_i(x) \leq 1\,,  & x\in\Omega_i\\
\eta_i(x) = 0\,,  & x\in\Omega\backslash\Omega_i.
\end{cases}
\]
The collection of local boundary conditions $f=[f_1(x)\,,\dotsc
  f_N(x)]$ are the unknowns in the iteration, found by updating
solutions in local patches iteratively until adjacent patches coincide
in the regions of overlap. Upon finding $f$, one finds the entire
global solution using~\eqref{eqn:global}.

The Schwarz method starts by assigning initial guesses to the local
boundary conditions $f_i(x)$, then solves all subproblems
\eqref{eqn:elliptic_local}, possibly in parallel. The local solutions
$u_i^\vep(x)$ so obtained are then used to update local boundary
conditions for those neighboring subdomains whose boundaries are
contained in $\Omega_i$ (that is, the subdomains denoted by $\Ical_i$
in \eqref{eqn:ical}). The procedure can be summarized as follows:
\begin{equation}\label{eqn:update2}
f_i(x) \xrightarrow{\Scal_i} u_i(x) \xrightarrow{\Pcal_{i}} f_{j}(x)\,, \text{ for all } j\in \Ical_i,
\end{equation}
where $\Scal_i$ the solution operator of equation
\eqref{eqn:elliptic_local} over subdomain $\Omega_i$ and $\Pcal_{i}$
is the restriction operator of the solution to the neighboring
boundaries $\partial\Omega_{j}$ ($j\in \Ical_i$). Defining the
boundary-to-boundary map by $\Acal_i := \Pcal_{i} \circ \Scal_i$, and
defining $\Acal$ to be the aggregation of $\Acal_i$ over
$i=1,2,\dotsc,N$, one can denote
\[
f^\text{new} = \Acal f^\text{old}\,.
\]
The overall procedure is summarized in \cref{alg:schwarz}.  The total
CPU time for this method is approximately the product of the number of
iterations $T$ and the CPU time $\tau$ for each iteration. While $T$
depends on the conditioning of the system, the value of $\tau$ is
determined by the cost of solving the local equation
\eqref{eqn:elliptic_local}.

\begin{algorithm}[h]
	\caption{Schwarz method for equation \eqref{eqn:elliptic}}\label{alg:schwarz}
	\begin{algorithmic}[1]
		\STATE Given total iterations $T$, boundary condition $f$;
		\STATE For $i=1,2,\dotsc,N$, initiate $f_i^0(x) = b(x)$ for $x\in \partial\Omega \cap \partial \Omega_i$ and assign $f_i^0(x)=0$ elsewhere;
                \STATE $t \leftarrow 0$;
		\WHILE{$t<T$}
		\FOR{$i=1,\ldots,N$} 
		\STATE Load $f_i^t$ as boundary condition for equation \eqref{eqn:elliptic_local} and solve for $u_i^t(x)$;
		\STATE Update $f_{j}^{t+1}(x) = u_i^t(x)\,, x\in \Omega \cap \partial \Omega_{j}$ for all $j\in \Ical_i$;
		\ENDFOR
		\ENDWHILE
		\STATE For $i=1,2,\dotsc,N$, load $f_i^T$ as boundary condition for equation \eqref{eqn:elliptic_local} and solve for $u^T_i(x)$;
		\STATE Assemble global solution $u^T(x) = \sum_{i=1}^{N} \eta_i(x)u^T_i(x)$;
		\RETURN $u^T(x)$.
	\end{algorithmic}
\end{algorithm}

\section{Schwarz method based on random sampling}\label{sec:randomSampling}
%%%%%%%%%%%%%%%%%%%%%%%%%%%%%%%%%%%%%%%%%%%
% section 3 outline:
%1. Random sampling
%2. Schwarz method based on random sampling
%%%%%%%%%%%%%%%%%%%%%%%%%%%%%%%%%%%%%%%%%%%

In this section, we propose a new algorithm that incorporates random
sampling into Schwarz iteration, under the domain decomposition
framework. As discussed above, each iteration requires the solution of
many subproblems, so the cost of obtaining local solutions is critical
to the overall run time. The key operation is the boundary-to-boundary
map $f^{t+1} = \Acal f^t$, which maps the boundary conditions
for the patches at iteration $t$ of the Schwarz procedure to an updated
set of boundary conditions obtained by solving the subproblems on the
patches, then restricting the solution to the patch boundaries. This
map can be prepared offline.  Moreover, by noting that the equation is
``homogenizable'' to an effective equation, we show that the
boundary-to-boundary map has approximately low rank. Techniques
inspired by from randomized linear algebra can be used to find the
low-rank approximation efficiently.

Our main tool is the randomized SVD algorithm. It is was shown
in~\cite{Tropp11} that by multiplying a low-rank matrix and its
transpose by several i.i.d. Gaussian vectors, and performing several
other inexpensive operations, the rank information could be captured
with high accuracy and high probability. We translate this technique
to the PDE setting and use it to reduce the cost of the
boundary-to-boundary map.

Since the random sampling technique plays a crucial role in finding
the approximated map, we quickly review the randomized SVD algorithm
of \cite{Tropp11} in \cref{subsec:randomSampling}. This algorithm
requires a matrix-vector operation involving the adjoint matrix, and
in the PDE setting, we need to find the adjoint associated with the
boundary-to-boundary map, an operation discussed in
in~\cref{subsec:adjointMap}. At the end of this section, we integrate
all components and summarize the algorithm.

\subsection{Random sampling for low rank matrices}
\label{subsec:randomSampling}

Random sampling algorithms have been widely used in numerical linear
algebra and machine learning. They are powerful in extracting
efficiently the main features of objects whose intrinsic dimension is
much smaller than their apparent dimension, such as sparse vectors,
low-rank matrices, or low-dimensional manifolds. We review the
randomized SVD algorithm applied on a large matrix $\Asf\in
\mathbb{R}^{m\times n}$ with $m\geq n$. The SVD of $\Asf$ is 
\begin{equation*}
	\Asf = \Usf \Sigma \Vsf^\top = \sum_{i=1}^n \sigma_i u_i v_i^\top\,,
\end{equation*}
where $\Usf =[u_1,\ldots,u_n] \in \mathbb{R}^{m\times n}$ contains the
left singular vectors, $\Vsf = [v_1,\ldots,v_n]$ contains the right
singular vectors and $\Sigma = \text{diag}(\sigma_1,\ldots,\sigma_n)$
contains the singular values in decreasing order: $\sigma_1\geq
\sigma_2\geq \ldots \sigma_n\geq 0$. If $\Asf$ has approximate rank
$k$, the best rank-$k$ approximation of $\Asf$ is the truncated
rank-$k$ singular value decomposition 
\begin{equation*}
	\Asf_k := \Usf_k \Sigma_k \Vsf^\top_k = \sum_{i=1}^k \sigma_i u_i v_i^\top\,,
\end{equation*}
where $\Usf_k$ and $\Vsf_k$ collect the first $k$ columns of $\Usf$
and $\Vsf$, respectively, and $\Sigma_k$ is the principal $k \times k$
major of $\Sigma$.  The relative error of this approximation is given
by
\begin{equation*}
	\frac{\| \Asf - \Asf_k\|_2}{\|\Asf\|_2} = \frac{\sigma_{k+1}}{\sigma_1}\ll 1\,.
\end{equation*}
Computation of the singular value decomposition of $\Asf$ requires
$\mathcal{O}(mn^2)$ time, which is expensive for large $m$ and $n$. The
randomized SVD algorithm computes an approximation to $\Asf_k$ by
simply applying the matrix $\Asf$ to a relatively few random
i.i.d. Gaussian vectors. The prototype randomized SVD algorithm is
shows as~\cref{alg:rsvdMatrix}.

\begin{algorithm}[h]
	\caption{Randomized SVD algorithm}\label{alg:rsvdMatrix}
	\begin{algorithmic}[1]
		\STATE Given an $m\times n$ matrix $\Asf$, target rank $k$;
		\STATE \textbf{Stage A:}
		\STATE Generate an $n\times 2k$ Gaussian test matrix $\Omega$;
		\STATE Form $\Ysf =  \Asf \Omega$;
		\STATE Perform the QR-decomposition of $\Ysf$: $\Ysf=\Qsf\mathsf{R}$
		\STATE \textbf{Stage B:}
		\STATE Form $\Bsf =  \Asf^\top \Qsf$;
		\STATE Compute the SVD of the $2k \times n$ matrix $\Bsf^\top = \wt{\Usf}\Sigma \Vsf^\top$;
		\STATE Set $\Usf = \Qsf \wt{\Usf}$;
		\RETURN $\Usf, \Sigma, \Vsf$.
	\end{algorithmic}
\end{algorithm}

For completeness, we give the error estimate result.
\begin{theorem}(\cite[Theorem~10.8]{Tropp11})
~\cref{alg:rsvdMatrix} finds accurate SVD of $\Asf$ in
expectation:
\begin{equation*}
	\mathbb{E} \| \Asf - \Usf \Sigma\Vsf^\ast \| \leq \left[1 + 4\sqrt{\frac{2\min\{m,n\}}{k-1}}\right]\sigma_{k+1}\,,
\end{equation*}
\end{theorem}
\cref{alg:rsvdMatrix} requires time complexity
\begin{equation}\label{eqn:timeComplexity}
	T_\text{randSVD} = 2k T_\text{mult} + \mathcal{O}(k^2(m+n))\,,
\end{equation}
where $2kT_\text{mult}$ is the time complexity of matrix-vector multiplication with
$\Asf$ and $\Asf^\top$.  
We note that assuming $\Asf$ is approximately of low rank $k\ll \min\{m,n\}$, the complexity
is rather low, and with fast decaying $\sigma_k$, the error is small too.

\subsection{Adjoint map}\label{subsec:adjointMap}
We aim at integrating the randomized SVD algorithm into the framework
of Schwarz method. As described in~\cref{alg:schwarz}, each
time step amounts to an update of the boundary conditions $f$ on the
patches, and has the form
\begin{equation}\label{eqn:update5}
f^{t+1} = \Acal f^t(x)\,,\quad \text{with}\quad\Acal_i = \Pcal_{i} \circ \Scal_i\,.
\end{equation}
where $\Scal_i$ is a solution operator and $\Pcal_i$ is a restriction
operator, both discussed further below. Since the equation is
homogenizable, many degrees of freedom can be neglected, making
$\Acal_i$ approximately low rank.~\cref{alg:rsvdMatrix} is
therefore relevant, but there is an immediate difficulty.  After
applying the full matrix (or operator) to some random vectors in Stage
A, we need in Stage B to apply the {\em transpose} or {\em adjoint} to
given vectors $q_i$.  For the operator define in \eqref{eqn:update5},
we need to know how to operate with both $\Acal_i\xi$ and
$\Acal^\ast_i \zeta$ for any given $\xi$ and $\zeta$.  Computing
$\Acal_i\xi$ is rather straightforward: it amounts to set local
boundary condition being $\xi$ and find the solution's confinement on
the neighboring cells' boundaries. Operating with the adjoint
$\Acal_i^\ast$ is somewhat more complicated.

A second difficulty has to do with the nature of the low rank of
$\Acal_i$. In general, the solution map $\Scal_i$ does {\em not} have
low rank, as we see in Section~\ref{sec:numerics}. We can however identify a {\em confined} solution map $\wt{\Scal}_i$,
which maps the boundary condition on $\partial\Omega_i$ to the {\em
  interior} solution $\widetilde{\Omega}_i$. By composing with the
restriction operator $\Pcal_{i}$ (slightly redefined), we obtain the same $\Acal_i$ of
\eqref{eqn:update5}. It happens that this confined operator
$\wt{\Scal}_i$ has approximately low rank.

Specifically, we define $\wt{\Scal}_i$ as follows: Given the boundary
condition $f$ over $\partial \Omega_i$, we have $ \wt{\Scal}_i f =
u|_{\widetilde{\Omega}_i} $, where $u$ solves the system
\begin{equation*}
\nabla\cdot (a^\vep(x)\nabla u(x) )= 0\,, \text{ in } \Omega_i\quad\text{with}\quad u(x) = f(x) \,, \text{ on } \partial \Omega_i\,.
\end{equation*}
We then rewrite \eqref{eqn:update5} as follows:
\begin{equation}\label{eqn:update6}
f^{t+1} = \Acal f^t(x)\,,\quad \text{with}\quad\Acal_i = \Pcal_{i} \circ \wt{\Scal}_i\,.
\end{equation}

To find the adjoint of $\wt{\Scal}_i$ we show the following theorem.

\begin{theorem}\label{thm:ad}
	Given two open sets $\wt{\Omega}$ and $\Omega$ such that $\overline{\wt{\Omega}} \subset \Omega$, then for arbitrary $f\in H^{1/2}(\partial\Omega)$ and $g\in H^1(\wt{\Omega})$, we have
	\begin{equation}\label{eqn:pf0}
	\langle g,\wt{\Scal} f\rangle_{\wt{\Omega}} =  \langle \wt{\Scal}^\ast g, f \rangle_{\partial \Omega},
	\end{equation}
	where $\wt{\Scal}$ and $\wt{\Scal}^\ast$ are defined as follows:
	\begin{equation}\label{eqn:thm}
	\begin{aligned}
	\wt{\Scal}: &\quad H^{1/2}(\partial \Omega) &\rightarrow &\quad H^1(\wt{\Omega}) \\
	&\quad f &\mapsto &\quad u|_{\wt{\Omega}},
	\end{aligned}
	\end{equation}
	where $u$ is the solution of the following elliptic equation:
	\begin{equation}\label{eqn:thm2}
	\begin{cases}
	\nabla\cdot (a(x)\nabla u(x)) = 0\,, & \mbox{\rm in $\Omega$} \\
	u(x) = f(x)\,, & \mbox{\rm on $\partial \Omega$},
	\end{cases}
	\end{equation}
	and 
	\begin{equation}\label{eqn:thm3}
	\begin{aligned}
	\wt{\Scal}^\ast: &\quad H^{1}(\wt{\Omega}) &\rightarrow &\quad H^{-1/2}(\partial \Omega) \\
	&\quad g &\mapsto &\quad a \frac{\partial v}{\partial n},
	\end{aligned}
	\end{equation}
	where $v$ solves the following sourced elliptic equation:
	\begin{equation}\label{eqn:thm4}
	\begin{cases}
	\nabla\cdot (a(x)\nabla v(x)) = \wt{g}\,, & \mbox{\rm in $\Omega$} \\
	v(x) = 0(x)\,, & \mbox{\rm on $\partial \Omega$,} 
	\end{cases}
	\end{equation}
	and $\wt{g}$ is the zero extension of $g(x)$ over $\Omega$. 
\end{theorem}
\begin{proof}
	Notice that the term on the left in \cref{eqn:pf0} is
	\begin{equation}\label{eqn:pf1}
	\langle g,\wt{\Scal} f\rangle_{\wt{\Omega}} = \int_{\wt{\Omega}} g u = \int_\Omega \wt{g}u = \int_\Omega u \nabla \cdot (a\nabla v).
	\end{equation}
Here $\langle \cdot,\cdot \rangle_{\wt{\Omega}}$ denotes the $L^2$
pairing over $\wt\Omega$ and the second and third equality comes from
the definition of $\wt{g}$. Applying Green's second identity, we obtain
\begin{equation}\label{eqn:pf2}
\int_\Omega \left[ u \nabla \cdot (a \nabla v) - v \nabla\cdot (a \nabla u)\right ] = \int_{\partial \Omega} a\left( u \frac{\partial v}{\partial n} - v \frac{\partial u}{\partial n}\right).
\end{equation}
By comparing \eqref{eqn:pf1} and \eqref{eqn:pf2}, we have
\begin{align*}
\langle g,\wt{\Scal} f \rangle_{\wt{\Omega}} &= \int_\Omega v\nabla\cdot (a\nabla u) + \int_{\partial \Omega} a\left( u \frac{\partial v}{\partial n} - v \frac{\partial u}{\partial n}\right) \\
	& = \int_{\partial \Omega} u \left( a\frac{\partial v}{\partial n} \right) \\
	& = \int_{\partial\Omega} f \left( a\frac{\partial v}{\partial n} \right) \\
	& = \langle \wt{\Scal}^\ast g, f \rangle_{\partial\Omega},
\end{align*}
where we use \eqref{eqn:thm2} and \eqref{eqn:thm4}. Thus
\eqref{eqn:pf0} is proved.
\end{proof}

This theorem shows how to evaluate the adjoint operator
$\wt{\Scal}^{\ast}_i$, thus making it possible to adapt the randomized
SVD approach of \cref{alg:rsvdMatrix} to our setting.  We summarize
the resulting method as~\cref{alg:rsvd}. It requires only $k$
solves of local elliptic PDE \eqref{eqn:elliptic_local} and sourced
elliptic PDE \eqref{eqn:thm4}, together with a QR factorization and
SVD of relatively small matrices.

\begin{algorithm}[h]
	\caption{Randomized SVD for $\wt{\Scal}_i$}\label{alg:rsvd}
	\begin{algorithmic}[1]
		\STATE Given target rank $k$ (an even number) and numerical solver for \eqref{eqn:elliptic_local} and \eqref{eqn:thm4};
		\FOR{$j=1,\ldots,k$}
		\STATE Generate random boundary conditions $\xi_j$ over $\partial \Omega_i$;
		\STATE Load $\xi_j$ as boundary condition in \eqref{eqn:elliptic_local}, and solve to obtain $u_j$;
		\STATE Take restrictions of $u_j$ over $\wt{\Omega}_i$ to obtain $\wt{u}_j$;
		\ENDFOR
		\STATE Find orthonormal basis $Q = [q_1,\ldots,q_k]$ of $\wt{U}:=\{\wt{u}_1,\ldots,\wt{u}_k\}$;
		\FOR{$j=1,\ldots,k$}
		\STATE Construct zero extension of $q_k$ over $\Omega_i$, denoted by $\wt{q}_k$;
		\STATE Load $\wt{q}_k$ as source in \eqref{eqn:thm4}, and solve to obtain $v_j$;
		\STATE Compute boundary flux $b_j := a\frac{\partial v_j}{\partial n}$;
		\ENDFOR
		\STATE Assemble all fluxes $B = [b_1,\ldots,b_k]$;
		\STATE Compute SVD of $B^\ast = \wt{\Usf}_k\Sigma_k \Vsf_k^\ast$;
		\STATE Compute $\Usf_k = Q \wt{\Usf}_k$;
		\RETURN $\Usf_k,\Sigma_k,\Vsf_k$.
	\end{algorithmic}
\end{algorithm}

This procedure can be executed offline to produce a rank-$k$
approximation to $\wt{\Scal}_i$. In online execution of the Schwarz
iteration procedure, \cref{alg:schwarz}, the update procedure
\eqref{eqn:update2} can be modified by replacing $\wt{\Scal}_i$ with
its low rank approximation, defined as follows:
\begin{equation}\label{eqn:update4}
f_i(x) \xrightarrow{\Usf_k\Sigma_k\Vsf_k^\ast} \wt{u}^t_i(x) \xrightarrow{\Pcal_{i}} f_{j}(x)\,, \text{ for all } j\in \Ical_i.
\end{equation}

We summarize our reduced Schwarz procedure as \cref{alg:schwarz_red}.

\begin{algorithm}[ht]
	\caption{Reduced Schwarz method for \eqref{eqn:elliptic}}\label{alg:schwarz_red}
	\begin{algorithmic}[1]
		\STATE Given rank $k$, total iterations $T$, boundary condition $b$;
		\STATE \textbf{Offline}:
		\FOR{$n=1,\ldots,N$}
		\STATE Use \cref{alg:rsvd} to find the rank-$k$ RSVD of $\wt{\Scal}_i$, denoted by $\Usf_k^i\Sigma_k^i\Vsf_k^{i,\ast}$;
		\ENDFOR
		\STATE \textbf{Online}:
		\STATE Initiate $f_i^0(x) = b(x)$ for $x\in \partial\Omega \cap \partial \Omega_i$ and assign $f_i^0(x)=0$ elsewhere;
		\WHILE{$t<T$}
		\FOR{$n=1,\ldots,N$}
		\STATE Evaluate $\wt{u}_i^t = \Usf_k^i\Sigma_k^i\Vsf_k^{i,\ast} f_i^t$;
		\STATE Update $f_{j}^{t+1}(x) = \wt{u}_i^t(x)\,, x\in \Omega \cap \partial \Omega_{j}$, for all $j\in \Ical_i$;
		\ENDFOR
		\ENDWHILE
		\FOR{$n=1,\ldots,T$}
		\STATE Load $f_i^T$ as boundary condition for equation \eqref{eqn:elliptic_local} and solve for $u^T_i(x)$;
		\ENDFOR
		\STATE Assemble global solution $u^T(x) = \sum_{n=1}^{N} \eta(x)u^T_i(x)$;
		\RETURN $u^T(x)$;
	\end{algorithmic}
\end{algorithm}

\section{Numerical Experiments} \label{sec:numerics}

In this section, we report on several numerical tests that demonstrate
effectiveness of our algorithm. We consider 
\eqref{eqn:elliptic} with a highly oscillatory media $a^\vep(x,y)$:
\[
a^\vep(x,y) = \frac{2+1.8\sin(\pi x/\vep)}{2+1.8\cos(\pi y/\vep)}+ \frac{2+\sin(\pi y/\vep)}{2+1.8\sin(\pi x)}\,, \quad (x,y)\in\Omega = [0,10] \times [0,1]\,,
\]
with $\vep = {1}/{16}$. This media is plotted in~\cref{fig:media}.
\begin{figure}[htb]
	\centering
	\includegraphics[width=0.7\textwidth]{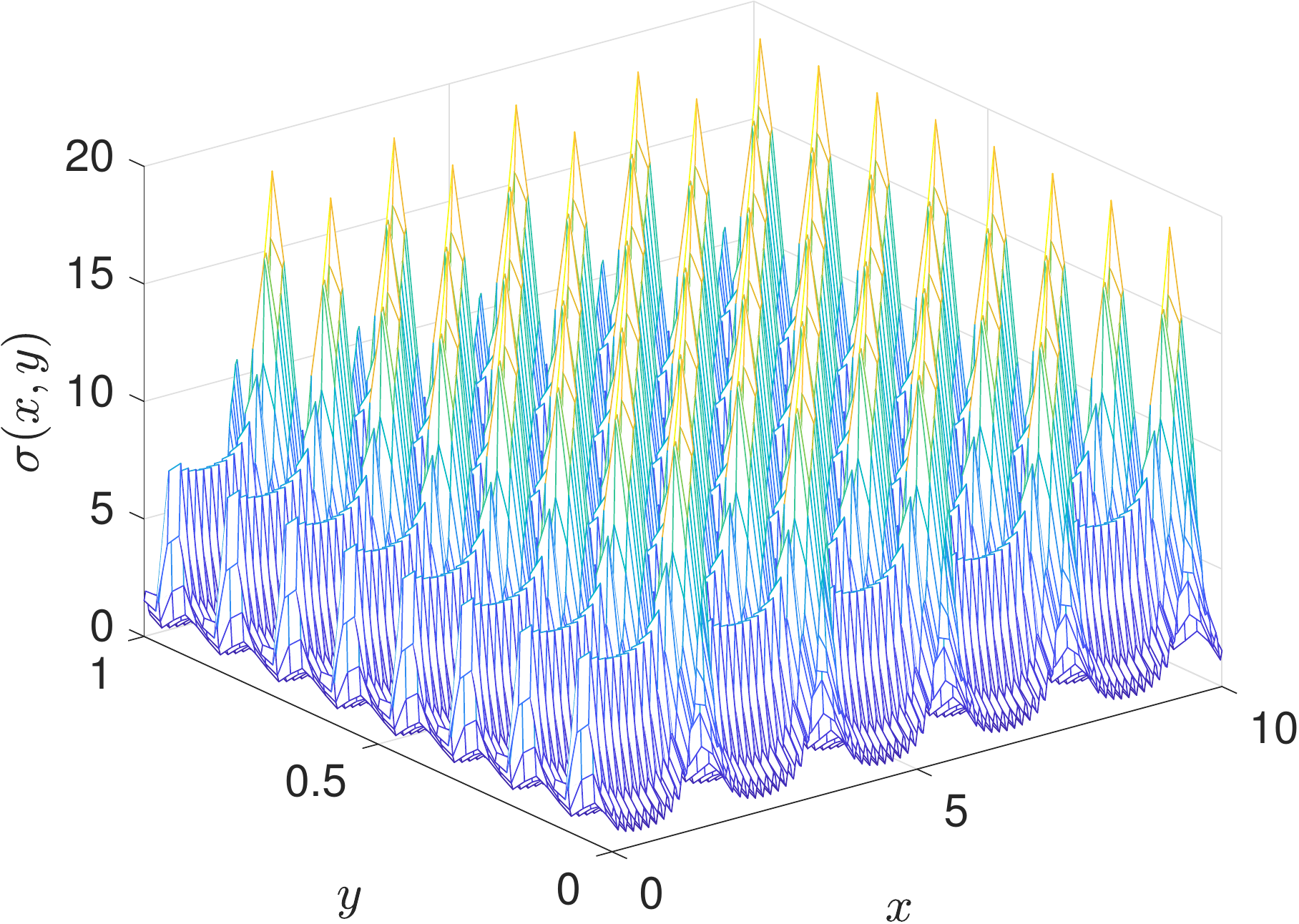}
	\caption{Graph of highly oscillatory media $a^\vep(x,y)$}
	\label{fig:media}
\end{figure}

To resolve the small scale, the fine discretization parameter
is set to $h = 1/40$ in both $x$ and $y$ direction. For ease of
implementation, we decompose the domain into subdomains in just one
dimension, as follows:
\[
\Omega  = \cup_{i=0}^{12} \Omega_n\,, \quad\text{with}\quad \Omega_i := \left[\frac{3i}{4},1+\frac{3i}{4} \right] \times [0,1]\,.
\]
Each subdomain $\Omega_i$ is thus a unit square with one quarter
margin overlapped with its neighbors on both sides. For this case, we
have that $\Ical_i = \{i-1,i+1\}$ for all inner patches;
see~\cref{fig:dd1}.  The boundary condition is
\[
b(x,y) = \sin\left(\frac{\pi}{3}(x-1/3)\right)	\sin\left(3\pi (y - 1/4)\right)\,,\quad\text{with}\quad (x,y)\in \partial \Omega\,.
\]

\begin{figure}[htb]
	\centering
	\includegraphics[width=0.8\textwidth]{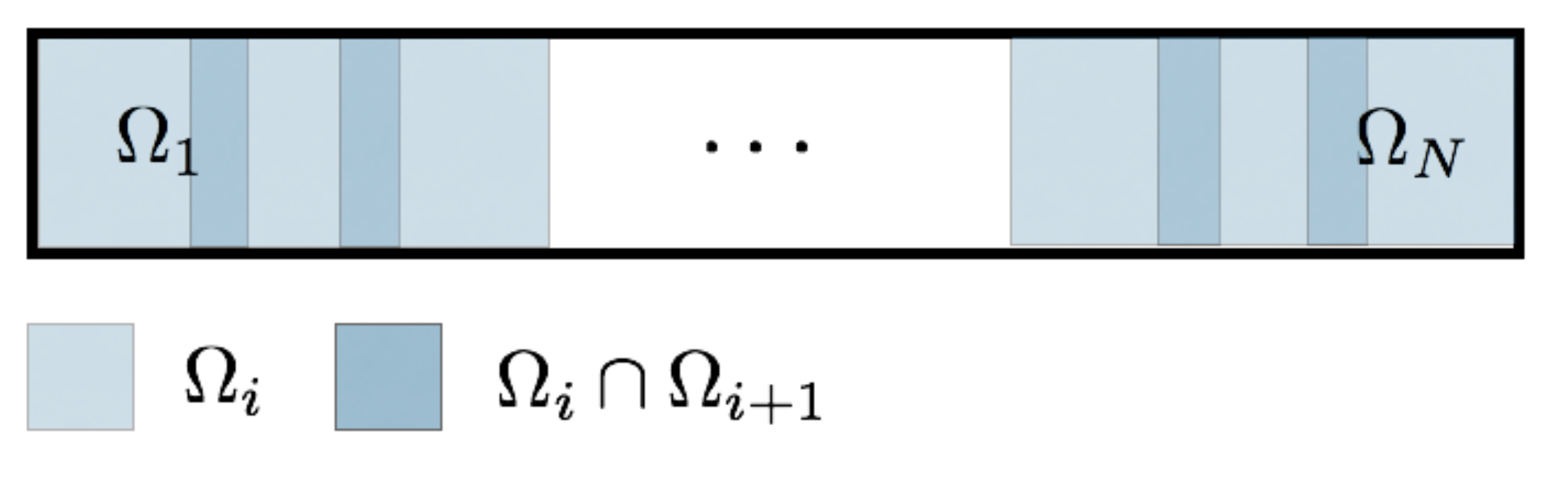}
	\caption{An overlapping domain decomposition of $\Omega$. Each subdomain $\Omega_i$ overlaps with its neighbors with one quarter margin.}
	\label{fig:dd1}
\end{figure}

In the next two subsections, we discuss the results of the offline
operation (low-rank approximation of the boundary-to-boundary map) and
the online iteration results, respectively.

\subsection{Reducibility of update procedure}

As described above, the mapping $\Acal_i$ is composed of a
boundary-to-solution map composed with a trace-taking operation,
defined by either \eqref{eqn:update5} or \eqref{eqn:update6}.  We
claimed above that the the map $\wt{\Scal}_i$ is approximately low
rank, while $\Scal_i$ is not. In \cref{fig:svdDecay}, we plot the
singular values of these two operators for subdomain $\Omega_3$, and
observe these claims to hold for this subdomain. We also plot the
singular values of $\Acal_3$, for which the low-rank structure
is even more evident. Similar results hold for the other inner subdomains $\Omega_i$,
$i=1,2,\dotsc,11$.

\begin{figure}
	\centering
	\includegraphics[width=0.7\textwidth]{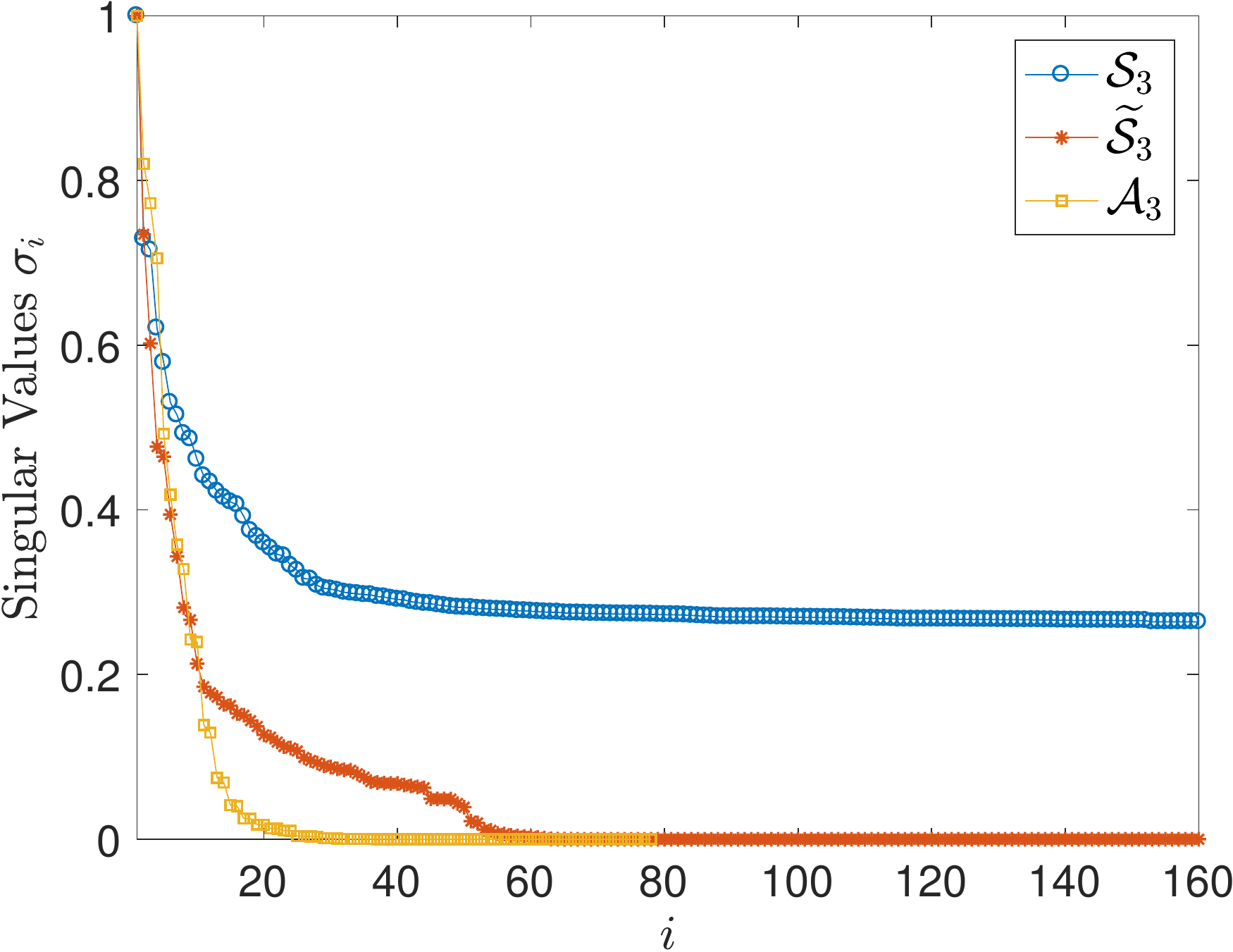}
	\caption{Singular values of operator $\Scal_3,\wt{\Scal}_3$ and $\Acal_3$. }
	\label{fig:svdDecay}
\end{figure}

\subsection{Performance of reduced Schwarz method}
To demonstrate the accuracy and efficiency of our method, we run the
reduced Schwarz method, \cref{alg:schwarz_red}, for several values of
the rank parameter ($k = 40,70,100,130$) in each subdomain. The reference solution $u_{\text{ref}}$ is computed
using the vanilla Schwarz method with $T=100$, at which iteration the
relative difference between successive iterations reaches machine
precision. In \cref{fig:soln}, we
compare the reference solution to the solution produced by
\cref{alg:schwarz_red} with $k=40$ and $T=50$. The difference is
barely visible.

\begin{figure}[htb]
	\centering
	\includegraphics[width=0.45\textwidth]{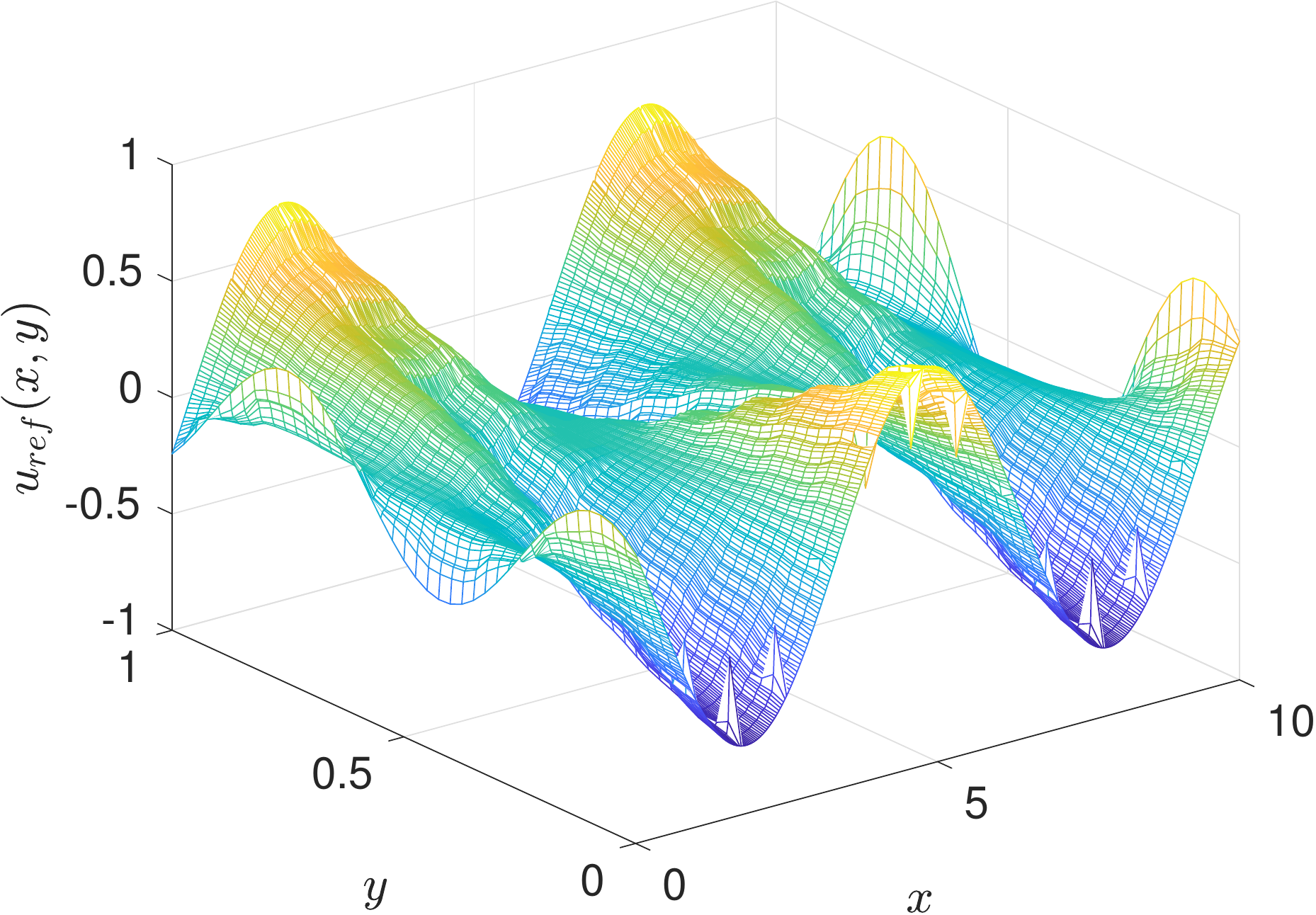}
	\includegraphics[width=0.45\textwidth]{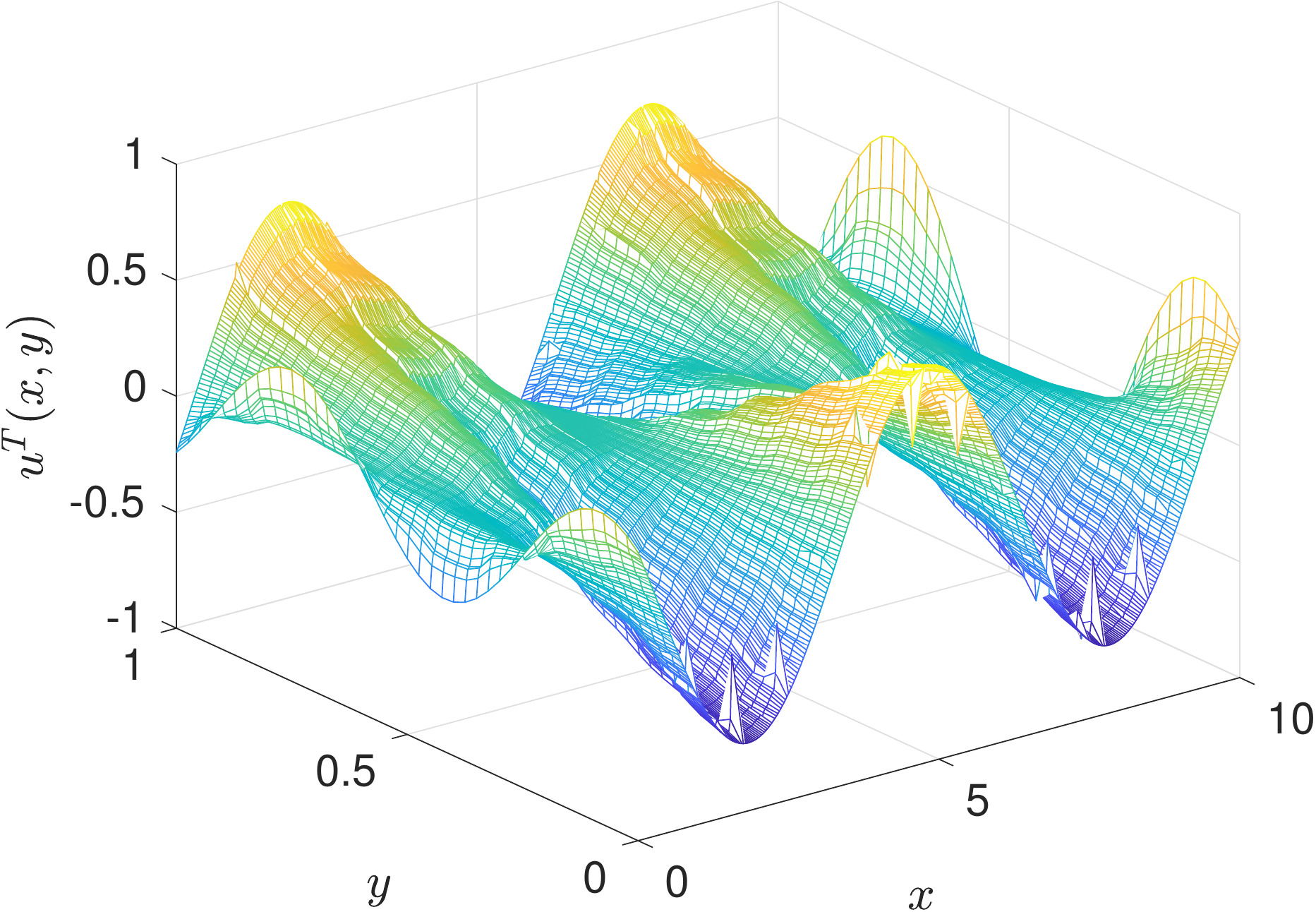}
	\caption{\textbf{Left:} Reference solution generated by running the Schwarz method with $100$ iterations; \textbf{Right:} Approximate solution generated by running the reduced Schwarz method with rank $k=40$ for $50$ iterations.}
	\label{fig:soln}
\end{figure}

We  document the relative errors defined by
\[
\text{Relative Error} = \frac{\| u - u_\text{ref}\|_2}{\|u_\text{ref}\|_2}
\]
at the iterates of both the vanilla Schwarz (\cref{alg:schwarz}) and the
reduced Schwarz (\cref{alg:schwarz_red}) methods, the latter for
various values of $k$. From this semilog plot, it is clear that error
decays exponentially with iteration number in all cases, and at the
same rate. Moreover, increasing $k$
allows the error to saturate at an increased level of accuracy. Still,
with rank $k=70$ (just one third of the full basis), we are already
able to capture an accuracy of $10^{-5}$.

\begin{figure}[htb]
	\centering
	\includegraphics[width=0.7\textwidth]{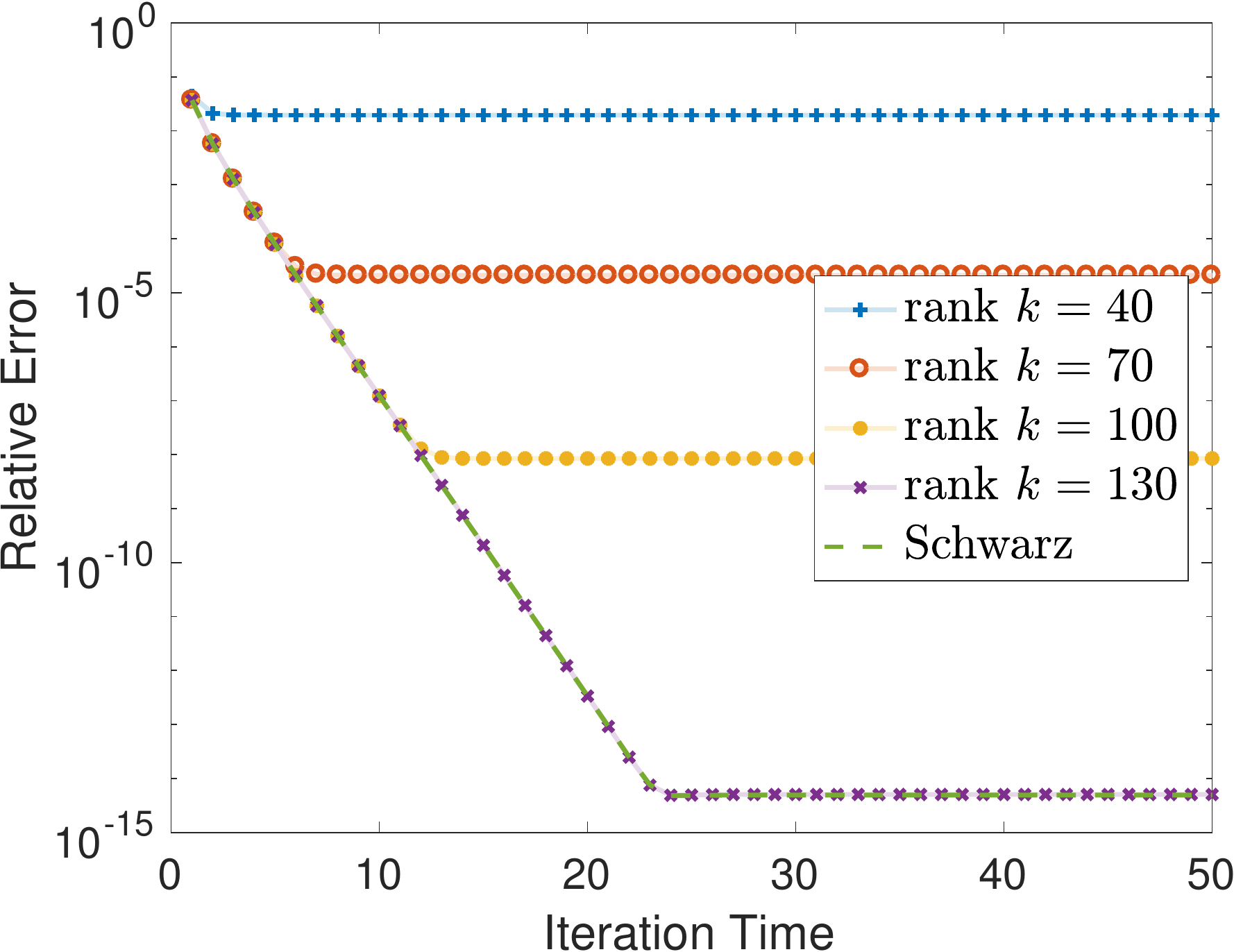}
	\caption{Relative Error for Schwarz method and reduced Schwarz methods, for various ranks $k$. The relative error decreases exponentially in time, and saturates at different levels. A higher rank $k$ in the reduced Schwarz method causes the error to saturate at a higher level of accuracy.}
	\label{fig:relErr}
\end{figure}

To demonstrate efficiency, we report in \cref{tbl:time} offline and
online calculation time for reduced Schwarz method, and compare it
with the vanilla Schwarz method. Although the reduced method is slower
overall for solving a single instance, it is extremely fast in the
online stage, implying that it is highly competitive if one needs to
solve \eqref{eqn:elliptic} for multiple different boundary
conditions. (In this case, each new set of boundary conditions
requires only the online stage to be performed in the reduced Schwarz
procedure, whereas in vanilla Schwarz, the entire method needs to be
executed again.)

\begin{table}[htb]
	\centering
	\begin{tabular}{l | c | c | c | c| c}
		\hline 
		\multirow{2}{*}{Run Time (s)}&\multicolumn{4}{ c |}{Reduced Schwarz }& \multirow{2}{*}{Schwarz}\\
		& $k = 40$&$k = 70$&$k = 100$&$k = 130$&  \\
		\hline
		Offline Stage	&49.7   &87.3  &129.4  &167.4 & 0 \\
		Online Stage  &.049  &  .061  &  .070   & .068 &31.4 \\
		Total Time	&49.8  & 87.3 & 129.4  &167.4 & 31.4\\
		\hline
	\end{tabular}
	\caption{Run times for the vanilla Schwarz method and the reduced Schwarz method for several values of $k$.}
	\label{tbl:time}
\end{table}

\section*{Acknowledgments}
We would like to acknowledge anonymous referees for suggestions, and Thomas Y. Hou and Jianfeng Lu for insightful discussions.

\bibliographystyle{siamplain}
\bibliography{references}
\end{document}